\documentclass
 {amsart}
\usepackage{amsmath,amssymb}

\newtheorem{propo}{{\bf Proposition}}[section]
\newtheorem{coro}[propo]{{\bf Corollary}}
\newtheorem{lemma}[propo]{{\bf Lemma}}
\newtheorem{remark}[propo]{{\bf Remark}}
\newtheorem{theor}[propo]{{\bf
Theorem}} 
\newtheorem{definition}{\bf Definition}

\usepackage{comment}
\usepackage{array}
\usepackage{mathrsfs}
\usepackage{color}
 \usepackage{lmodern}
\usepackage[T1]{fontenc}
\usepackage{a4wide}
\usepackage{enumitem}
\setlist[enumerate,1]{label={\roman*)}}

\begin{document}

\title[Poisson closures and Poisson $A$-algebras]{Poisson closures and Poisson $A$-algebras}

\bigskip

\author[Amir Fernández Ouaridi]{Amir Fernández Ouaridi}
\address{Amir Fernández Ouaridi. \newline \indent University of Sevilla, Department of Geometry and Topology, Sevilla (Spain).}
\email{{\tt amir.fernandez.ouaridi@gmail.com}}

\author[Salvatore Siciliano]{Salvatore Siciliano}
\address{Salvatore Siciliano. \newline \indent University of Salento, Dipartimento di Matematica e Fisica "Ennio De Giorgi", Lecce (Italy).}
\email{{\tt salvatore.siciliano@unisalento.it}}

\author[David A. Towers]{David A. Towers}
\address{David A. Towers. \newline \indent School of
Mathematical Sciences, Lancaster University. 
Lancaster LA1 4YF (England).}
\email{{\tt d.towers@lancaster.ac.uk}}

\begin{abstract}
We study the interaction between the Lie and Poisson subalgebras of a Poisson algebra. Given a Lie subalgebra $U$ of the underlying Lie algebra $\mathcal P_L$, we consider the Poisson subalgebra generated by $U$, called its {Poisson closure}, and investigate when Lie-theoretic properties such as nilpotency and solvability are preserved under this construction. We also study the stability under Poisson closure of several distinguished Lie subalgebras, including maximal nilpotent and  solvable subalgebras, Engel and Cartan subalgebras and the derived algebra. As an application, we investigate Poisson $A$-algebras, defined by the condition that every nilpotent subalgebra is a zero algebra. In the finite-dimensional case, we obtain structural results relating Poisson $A$-algebras to Lie $A$-algebras and to the associative structure of the algebra.
\end{abstract}


\thanks{Supported by Junta de Andalucía through the project FQM-336.}

\maketitle

\section{Introduction}
\label{sec1}
Poisson algebras lie naturally at the intersection of associative and Lie algebra theory. They consist of a commutative associative algebra together with a Lie bracket satisfying the Leibniz rule, and arise in several areas of Mathematics and Mathematical Physics. Their double algebraic structure makes it natural to study them simultaneously from the associative, Lie-theoretic and genuinely Poisson points of view.

In recent years, increasing attention has been devoted to the structure of nilpotent and solvable Poisson algebras; see, for example \cite{GR08, oo, FNT24, nsfp}. One approach is to define nilpotency and solvability intrinsically, by means of series involving both products, as introduced in \cite{pan}. This raises a natural problem: to what extent can structural information about a Poisson algebra be recovered from its underlying associative algebra $\mathcal P_A$ and Lie algebra $\mathcal P_L$? Results in this direction show that, under suitable assumptions, there is a strong connection between Poisson-theoretic properties and the corresponding properties of the underlying Lie algebra.

The aim of this paper is to pursue this question at the level of subalgebras. If $U$ is a Lie subalgebra of $\mathcal P_L$, the associative subalgebra $S_U$ generated by $U$ is, by the Leibniz rule, also closed under the Lie bracket. Thus $S_U$ is the smallest Poisson subalgebra containing $U$, and we call it the \emph{Poisson closure} of $U$. 
Our main purpose is to understand how Lie-theoretic properties behave under Poisson closure. In particular, we study when nilpotency and solvability are preserved by Poisson closure and which distinguished Lie subalgebras are already Poisson subalgebras. We consider, among others, maximal nilpotent and solvable subalgebras, Cartan and Engel subalgebras, and the derived algebra of $\mathcal P_L$. The latter leads naturally to questions concerning nil and nilpotent associative subalgebras, closely related to earlier work on solvable Poisson algebras, see \cite{S20, SU21, SU21b}.

These ideas are then applied to the study of Poisson $A$-algebras, the Poisson analogue of Lie $A$-algebras. Recall that Lie $A$-algebras are Lie algebras whose nilpotent subalgebras are abelian \cite{Aalg}. In the Poisson setting, we require nilpotent subalgebras to be zero algebras. We investigate some basic properties of this class and obtain structural results in the finite-dimensional case, showing in particular the close relationship between Poisson $A$-algebras and their Lie counterparts.

The paper is organized as follows. Section~\ref{sec2} contains the preliminaries and fixes the notation and basic results used throughout the paper. Section~\ref{sec3} is devoted to Poisson closures and their interaction with nilpotency and solvability. In Section~\ref{sec4}, we study the stability under Poisson closure of several distinguished Lie subalgebras. Finally, Section~\ref{sec5} concerns Poisson $A$-algebras and their structure.

\section{Preliminaries} \label{sec2}
\medskip
A dialgebra \((\mathcal{A},\cdot,[\cdot,\cdot])\) is a vector space \(\mathcal{A}\) equipped with two bilinear products, \(\cdot\) and \([\cdot,\cdot]\), neither of which is assumed to be associative. Important classes of dialgebras include Lie--Yamaguti algebras, Gerstenhaber algebras, Nambu--Poisson algebras, Novikov--Poisson algebras, Gelfand--Dorfman algebras, and many others. Among them, Poisson algebras constitute the principal class considered in this paper.

\begin{definition}\rm
A {Poisson algebra} $\mathcal{P}$ is a dialgebra $(\mathcal{P}, \cdot, [\cdot, \cdot])$ such that $\mathcal{P}_A:=(\mathcal{P}, \cdot)$ is an associative commutative algebra, $\mathcal{P}_L:=(\mathcal{P}, [\cdot, \cdot])$ is a Lie algebra and they satisfy the compatibility identity for $x,y,z\in \mathcal{P}$ given by
\begin{equation*}
[x\cdot y, z]=[x,z]\cdot y + x\cdot[y,z] \quad \quad \textrm{(Leibniz rule).}
\end{equation*} For convenience, we will often omit the $\cdot$ and denote the associative product simply by juxtaposition. Unless otherwise stated, we will not assume that $\mathcal{P}$ has an identity element. 
\end{definition}

\begin{definition} For a Poisson algebra $\mathcal{P}$, we denote by $P_x$ and $Q_x$ the maps in $\textrm{End}(\mathcal{P})$ given by $P_x(y) = x\cdot y$ and $Q_x(y) = [x, y]$, for $x, y \in \mathcal{P}$. 
\end{definition}

\begin{definition} For a Poisson algebra $\mathcal{P}$ we recall the associative and Lie derived series
$$\mathcal{P}_A^{(0)}=\mathcal{P}, \quad \mathcal{P}_A^{(n+1)}=\mathcal{P}_A^{(n)}\cdot \mathcal{P}_A^{(n)}\quad \hbox{and}\quad \mathcal{P}_L^{(0)}=\mathcal{P} \quad \mathcal{P}_L^{(n+1)}=[\mathcal{P}_L^{(n)},\mathcal{P}_L^{(n)}].$$ Also, we recall the associative and the Lie lower annihilating series
$$\mathcal{P}_A^1=\mathcal{P}, \quad \mathcal{P}_A^{n+1}=\mathcal{P}_A^n\cdot \mathcal{P}\quad \hbox{and} \quad \mathcal{P}_L^1=\mathcal{P}, \quad \mathcal{P}_L^{n+1}=[\mathcal{P}_L^n,\mathcal{P}].$$ We say that $\mathcal{P}$ is associative (respectively, Lie) solvable if there exists $m > 0$ such that $\mathcal{P}_A^{(m)}=0$ (respectively, $\mathcal{P}_L^{(m)}=0$) . Similarly, we can define associative and Lie nilpotency.
Moreover, we define
$$\mathcal{P}^{(0)}:=\mathcal{P} \quad \quad \quad \mathcal{P}^{(n+1)} =  \mathcal{P}^{(n)}\cdot\mathcal{P}^{(n)} + [\mathcal{P}^{(n)}, \mathcal{P}^{(n)}],\quad \hbox{and}$$
$$\mathcal{P}^1:=\mathcal{P} \quad \quad \quad \mathcal{P}^{n+1} =\mathcal{P}^{n}\cdot\mathcal{P}+ [\mathcal{P}^{n}, \mathcal{P}].$$ We say that $\mathcal{P}$ is solvable (respectively, nilpotent) if there exists $m > 0$ such that $\mathcal{P}^{(m)}=0$ (respectively, $\mathcal{P}^m=0$); also, that $\mathcal{P}$ is a zero algebra if $\mathcal{P}^{(1)}=0$.
\end{definition}

    Each term of the two mixed series is a Poisson ideal of $\mathcal P$. Thus the preceding definitions refer to solvability and nilpotency in the category of Poisson algebras.

\begin{definition}
    A Poisson algebra is called a \emph{nilalgebra} whenever it is nil as an associative algebra; that is, for every $x$ one has $x^{n(x)}=0$ for some positive integer $n(x)$.
\end{definition}

We shall use the full centre of a Poisson algebra,
$$
Z(\mathcal P):=\{z\in\mathcal P:z\mathcal P=0=[z,\mathcal P]\}.
$$
When only the centre of the underlying Lie algebra is intended, we shall write $Z(\mathcal P_L)$.

\medskip

Unless otherwise specified, all results in this paper hold in arbitrary dimension. Throughout the paper, all vector spaces and algebras are defined over a fixed field $\mathbb F$. All linear maps are assumed to be $\mathbb F$-linear, and all tensor products are taken over $\mathbb F$. Further conventions will be stated explicitly when needed.

\section{Poisson closures}\label{sec3}

Let $\mathcal{P}$ be a Poisson algebra and let $U$ be a Lie subalgebra of
$\mathcal{P}_L$. We denote by $S_U$ the associative subalgebra of
$\mathcal{P}_A$ generated by $U$. Thus
$$S_U=\sum_{k\geq 1} U_A^k = U + U\cdot U + \ldots$$
The Leibniz identity implies that $S_U$ is again a Poisson subalgebra of
$\mathcal{P}$. Hence
$S_U$ is the minimal Poisson subalgebra of $\mathcal{P}$ containing $U$. We call $S_U$ the {\it Poisson closure} of $U$ in $\mathcal{P}$.
The aim of this section is to determine when Lie-theoretic properties of
$U$, such as nilpotency or solvability, are inherited by its Poisson
closure.  

\begin{lemma}
Let $U$ be a Lie subalgebra of $\mathcal{P}_L$. Then $S_U$ is the minimal
Poisson subalgebra of $\mathcal{P}$ containing $U$.
\end{lemma}
\begin{proof}
     Since $S_U$ is, by definition, an associative subalgebra of $\mathcal{P}_A$,
it remains to show that it is closed under the Lie bracket. By the Leibniz
identity, we have $[U_A^i,U_A^j]\subseteq U_A^{i+j-1}$ 
for all $i,j\geq 1$. It follows $[S_U,S_U]\subseteq S_U$.
Moreover, if $B$ is any Poisson subalgebra of $\mathcal{P}$ containing $U$,
then $B$ contains all associative products of elements of $U$. Thus
$S_U\subseteq B$. Therefore $S_U$ is the minimal Poisson subalgebra of
$\mathcal{P}$ containing $U$.
\end{proof}

\begin{lemma}\label{sub} Let $U$ be a maximal abelian subalgebra of $\mathcal{P}_L$. Then $U$ is a subalgebra of $\mathcal{P}$. 
\end{lemma}
\begin{proof} Note that $(U+U\cdot U)_L^2\subseteq U_L^2+U\cdot U_L^2+U_A^2\cdot U_L^2=0$, so $U+U\cdot U$ is an abelian subalgebra of $\mathcal{P}_L$. It follows from the maximality of $U$ that  $U\cdot U\subseteq U$, whence $U$ is a subalgebra of $\mathcal{P}$.
\end{proof}


\begin{lemma}\label{square} If ${\mathcal{P}}$ is a Poisson algebra with $\mathcal{P}_A$ nilpotent, then $\mathcal{P}_A^2$ is a nilpotent ideal of $\mathcal{P}$.
\end{lemma}
{
\begin{proof}
 Clearly $\mathcal{P}_A^2$ is associative nilpotent and an ideal of $\mathcal{P}$. We show that $\mathcal{P}_A^2$ is Lie nilpotent and the result will follow by \cite[Theorem 2.13]{oo}. 
Notice that the Leibniz rule implies that $[\mathcal{P}_A^r,\mathcal{P}_A^s]\subseteq \mathcal{P}_A^{r+s-1}$. 
Then we can prove by induction on  $k>1$ that 
$$(\mathcal{P}_A^2)^k_L = [(\mathcal{P}_A^2)^{k-1}_L, \mathcal{P}_A^2] \subseteq \mathcal{P}_A^{k+1}.$$
Indeed, the case $k=2$ is trivial. If it holds for $k=n$, then we have
$$(\mathcal{P}_A^2)^{n+1}_L = [(\mathcal{P}_A^2)^{n}_L, \mathcal{P}_A^2]\subseteq [\mathcal{P}_A^{n+1}, \mathcal{P}_A^2] \subseteq \mathcal{P}_A^{n+2}.$$

Hence, the associative nilpotency of  $\mathcal{P}_A^2$ implies the Lie nilpotency.
\end{proof}

}
\begin{lemma}\label{lemLs}
    Let $\mathcal{P}$ be a Poisson algebra with $\mathcal{P}_A$ nilpotent. Then any subalgebra of $\mathcal{P}$ which is Lie solvable (resp. Lie nilpotent) is solvable (resp. nilpotent). 
\end{lemma}
\begin{proof}
    Let $U$ be a Lie solvable subalgebra of $\mathcal{P}$. We claim that $U^{(k)}\subseteq U_L^{(k)} + U_A^2$. This clearly holds for $k=1$. Suppose that it holds for $k=n$. Then 
\begin{align*}
U^{(n+1)}&\subseteq [U_L^{(n)} + U_A^2,U_L^{(n)} + U_A^2]+(U_L^{(n)} + U_A^2)\cdot (U_L^{(n)} + U_A^2) \\
&\subseteq U_L^{(n+1)}+[U,U_A^2]+[U_A^2,U_A^2]+U_A^2 \\
&\subseteq U_L^{(n+1)}+U\cdot [U,U]+U\cdot U\cdot [U,U]+U_A^2 \\
&\subseteq U_L^{(n+1)}+U_A^2.
\end{align*} Now, there is $n$ such that $U_L^{(n)}=0$ and so $U^{(n)}\subseteq U_A^2$. But $U_A^2$ is nilpotent, by Lemma~\ref{square}, and hence solvable. It follows that $U$ is solvable.

If $U$ is a Lie nilpotent subalgebra of $\mathcal{P}$, then it follows that it is nilpotent by \cite[Theorem 2.13]{oo}, since it is associative nilpotent.
\end{proof}

{
\begin{propo}\label{charsolv}
    Let $\mathcal{P}$ be a Poisson algebra. Then $\mathcal{P}$ is solvable if and only if $\mathcal{P}_A$ is nilpotent and $\mathcal{P}_L$ is solvable.
\end{propo}}
\begin{proof}
    The forward implication is immediate. Conversely, if $\mathcal P_A$ is nilpotent and $\mathcal P_L$ is solvable, then using Lemma~\ref{lemLs} applied to $U=\mathcal P$, we deduce that $\mathcal P$ is solvable.
\end{proof}

\begin{coro} Let $\mathcal{P}$ be a Poisson algebra of finite dimension with $\mathcal{P}_A$ nilpotent. Then  $R_L$, the radical of $\mathcal{P}_L$, is a solvable ideal of $\mathcal{P}$. Likewise, $N_L$, the nilradical of $\mathcal{P}_L$, is a nilpotent ideal of $\mathcal{P}$.
\end{coro}
\begin{proof} First, note that $R_L\cdot \mathcal{P}\cdot \mathcal{P}\subseteq R_L\cdot \mathcal{P}$ and $$[R_L\cdot \mathcal{P},\mathcal{P}] \subseteq R_L\cdot \mathcal{P}_L^2+\mathcal{P}\cdot [R_L,\mathcal{P}]\subseteq R_L\cdot \mathcal{P},$$ so $R_L\cdot \mathcal{P}$ is an ideal of $\mathcal{P}$. Now $\mathcal{P}_A^2$ is nilpotent and hence solvable (see \cite[page 796]{oo}). It follows that $R_L\cdot \mathcal{P}$ is a solvable ideal of $\mathcal{P}$ and hence of $\mathcal{P}_L$. Thus, $R_L\cdot \mathcal{P}\subseteq R_L$, whence $R_L$ is an ideal of $\mathcal{P}$. In particular, $R_L$ is an ideal of $\mathcal{P}$ which is Lie solvable, so the result follows by Lemma~\ref{lemLs}. 

An analogous argument shows that the nilradical of $\mathcal{P}_L$ is a nilpotent ideal of $\mathcal{P}$.
\end{proof}

\begin{coro} 
Let $\mathcal{P}$ be a Poisson algebra of finite dimension with $\mathcal{P}_A$ nilpotent. Then  $R_L=R$, the radical of $\mathcal{P}$, and $N_L = N$, the nilradical of $\mathcal{P}$.
\end{coro}
\begin{proof}
By the preceding corollary, $R_L$ (respectively, $N_L$) is a solvable
(respectively, nilpotent) ideal of $\mathcal P$, so $R_L\subseteq R$ and
$N_L\subseteq N$. The reverse inclusions follow because every solvable
(respectively, nilpotent) Poisson ideal is Lie solvable (respectively, Lie
nilpotent).
\end{proof}

\begin{lemma}\label{actn}
    Let $\mathcal{P}$ be a Poisson algebra in which $\mathcal{P}_A$ is nilpotent. If $U$ is a nilpotent subalgebra of $\mathcal{P}_L$, then $Q_U$ acts nilpotently on $S_U$.
\end{lemma}
\begin{proof}
    Since $\mathcal{P}_A$ is nilpotent, the subalgebra of $\mathcal{P}_A$ generated by $U$ is $S_U= \sum^{n-1}_{k=1} U^k_A$,  where $n$ is the nilpotency index of $\mathcal{P}_A$. Let $m$ be the nilpotency index of $U_L$. Then we have the following equation (which can be proven by induction on $t$)
    $$Q_U^{t}(U^k)\subseteq\sum_{i_1+ \ldots + i_k = t} Q_U^{i_1}(U)\cdots Q_U^{i_k}(U).$$
    From here, if we choose $t\geq nm$, it is clear that $Q_U^{t}(U^k)=0$ for $1 \leq k\leq n-1$ as one of the indexes in each summand will be greater than $m$. Therefore, it follows that $Q_U^{t}(S_U)=0$.
\end{proof}

\begin{propo}\label{propnil}
    Let $\mathcal{P}$ be a Poisson algebra in which $\mathcal{P}_A$ is nilpotent. If $U$ is a nilpotent subalgebra of $\mathcal{P}_L$, then $S_U$ is nilpotent.
\end{propo}
\begin{proof}
    Since $S_U$ is associative nilpotent, it remains to show that $S_U$ is Lie nilpotent. 
It is enough to prove that for a suitable positive number $k$ we have $Q_{s_1}\ldots Q_{s_k}(S_U)=0$ for every $s_1, \ldots, s_k\in S_U$ with $s_i=u_{i, 1}\cdots u_{i, k_i}\in U_A^{k_i}$. Using the relation $Q_{xy}=P_yQ_x + P_xQ_y$ for all $x, y\in \mathcal{P}$, which  follows by the Leibniz rule, we obtain 
$$Q_{s_i} = \sum_{t=1}^{k_i} P_{u_{i,1}}\ldots \widehat{P_{u_{i,t}}} \ldots P_{u_{i, k_i}} Q_{u_{i, t}},$$
where $\widehat{P_{u_{i,t}}}$ means that the factor $P_{u_{i,t}}$ is omitted.
Therefore, $Q_{s_1}\ldots Q_{s_k}$ can be written as a sum of terms of the form $T = T_1 \ldots T_N$ where $T_i$ can be an operator in $P_U$ or in $Q_U$. We will show that each of these terms vanishes in $S_U$. Notice that $N\geq k$ and that the number of $Q$ operators in each summand is precisely $k$.

Let $n$ be the index of nilpotency of $(S_U)_A$ and let $m$ be such that $Q_U^{m}(S_U)=0$. This $m$ exists by Lemma~\ref{actn}.
Since $S_U$ is a subalgebra of $\mathcal{P}$, any combination of operators in $P_U$ or in $Q_U$  yields an element in $S_U$. Choose $k> nm$. Then either $T$ has at least $n$ operators in $P_U$ or $T$ has a consecutive sequence of $m$ operators in $Q_U$.

In the first case, we 
note that the Leibniz rule is equivalent to $P_{[x,y]} = Q_{x}P_y - P_{y}Q_{x}$. Suppose $T = T_1 \ldots T_{i-1} Q_{x_i} P_{x_{i+1}}T_{i+2} \ldots T_N$ for  $x_i, x_{i+1}\in U$. By this relation we have $$T= T_1 \ldots T_{i-1} P_{x_{i+1}} Q_{x_i} T_{i+2} \ldots T_N + T_1 \ldots T_{i-1} P_{[x_i, x_{i+1}]} T_{i+2} \ldots T_N.$$
    We call this process an inversion. Notice that $P_{[x_i, x_{i+1}]}$ still belongs to $P_U$ and that the number of $P$ maps in both terms is preserved in an inversion. Also, observe that, by applying inversions, $T$ can be written as the sum of a term of length $N$ of the form $P \ldots P Q \ldots Q$ (we call this form regular) and other terms of lower length for certain $t\geq 0$. By induction on the length, other terms can be written in regular form too, so $T$ is a combination of regular terms, and each of these terms has the same number of operators in $P_U$ as the term from which it is obtained. Therefore, if $T$ has at least $n$ operators in $P_U$, each of its regular summands has at least $n$ operators in $P_U$. But then they have $n$ consecutive operators in $P_U$, so their action of $S_U$ vanishes by the nilpotency of $\mathcal{P}_A$. Hence, we conclude $T(S_U) = 0$.
    
In the second case, we simply have $T(S_U) = 0$ because $Q_U$ acts nilpotently on $S_U$ and we have $m$ consecutive maps in $Q_U$. 
\end{proof}

\begin{remark}\label{remex1}
    The assumption on the nilpotency of $\mathcal{P}_A$ cannot be removed in Proposition~\ref{propnil}. Consider the polynomial algebra $\mathbb{C}[x, y]$ equipped with the standard hamiltonian Lie bracket 
    \begin{equation}\label{eqbr}
        [f, g] = \dfrac{\partial f}{\partial x}\dfrac{\partial g}{\partial y} - \dfrac{\partial g}{\partial x}\dfrac{\partial f}{\partial y}.
    \end{equation} This is a well-known Poisson algebra. The Lie algebra $U$ spanned by $\left\{1,x,y\right\}$ is nilpotent. In fact, it is isomorphic to the Heisenberg algebra. However, $S_U=\mathbb{C}[x,y]$ is neither Lie nilpotent ($Q_{xy}$ has eigenvector $x$ with non zero eigenvalue) nor associative nilpotent.
\end{remark}

\begin{propo}\label{propsolv}
    Let $\mathcal{P}$ be a Poisson algebra in which $\mathcal{P}_A$ is nilpotent. If $U$ is a solvable subalgebra of $\mathcal{P}_L$, then $S_U$ is solvable.
\end{propo}
\begin{proof}
    Since $S_U$ is associative nilpotent, we will prove that it is Lie solvable and the result will follow by Proposition~\ref{charsolv}. 
    We proceed by induction on the nilpotency index $m$ of $(S_U)_A$. For $m=2$, the result is trivial since $S_U = U$. Assume the result holds whenever the associative nilpotency index is smaller than $m$.
Suppose $(S_U)_A^m=0$ and $I = (S_U)_A^{m-1}\neq 0$ where $m\geq 3$. Notice that $[S_U, I] = [S_U, (S_U)_A^{m-1}] \subseteq I$ by the Leibniz rule, so $I$ is a Lie ideal of $S_U$.
Clearly, the subalgebra $I$ is abelian because we have $$[I, I]\subseteq [(S_U)_A^{m-1}, (S_U)_A^{m-1}] \subseteq (S_U)_A^{2m-3} = 0.$$ 

Now, $S_U/I$ is associatively generated by $(U+I)/I$, which is a solvable Lie subalgebra of $S_U/I$. Also,
$S_U/I$ has associative nilpotency index at most $m-1$. Hence, $S_U/I$ is Lie solvable by hypothesis. Therefore, since $I$ and $S_U/I$ are Lie solvable,  $S_U$ is Lie solvable too.  
\end{proof}

\begin{theor}\label{nilcharacterization}
Let $\mathcal{P}$ be a Poisson algebra. The following statements are equivalent.
\begin{enumerate}
    \item $\mathcal{P}_A$ is a nilalgebra;
    \item for every finite-dimensional nilpotent subalgebra $U$ of $\mathcal{P}_L$, the subalgebra $S_U$ is nilpotent;
    \item for every finite-dimensional solvable subalgebra $U$ of $\mathcal{P}_L$, the subalgebra $S_U$ is solvable.
\end{enumerate}
Moreover, if $\mathcal{P}$ is finite-dimensional, then these conditions are
equivalent to $\mathcal{P}_A$ being nilpotent.
\end{theor}
\begin{proof}
Assume first that $\mathcal{P}_A$ is a nilalgebra, and let
$U$ be a finite-dimensional Lie subalgebra of $\mathcal{P}_L$. Choose a basis
$u_1,\ldots,u_n$ of $U$. Since $\mathcal{P}_A$ is a nilalgebra, we have $u_i^{k_i} = 0$ for some $k_i$. Since $\mathcal{P}_A$ is commutative, every monomial in $u_1,\ldots,u_n$ of degree at least $N=1+\sum_{i=1}^n (k_i-1)$ is zero. Hence, the algebra $(S_U)_A$ is nilpotent. 

If $U$ is nilpotent as a Lie algebra, then applying
Proposition~\ref{propnil} to the Poisson algebra $S_U$ gives that $S_U$ is
nilpotent. If $U$ is solvable as a Lie algebra, then applying
Proposition~\ref{propsolv} to the Poisson algebra $S_U$ gives that $S_U$ is
solvable. Thus (i) implies both (ii) and (iii).

Conversely, suppose that $S_U$ is nilpotent (resp. solvable) for every finite-dimensional nilpotent (resp. solvable) Lie
subalgebra $U$ of $\mathcal{P}_L$. Let $x\in\mathcal{P}$. Then $\mathbb{F} x$ is an
abelian Lie subalgebra of $\mathcal{P}_L$. Under condition $ii)$, $S_{\mathbb F x}$ is nilpotent, and hence $(S_{\mathbb F x})_A$ is nilpotent. Under condition $iii)$, $S_{\mathbb F x}$ is solvable, so $(S_{\mathbb F x})_A$ is nilpotent by Proposition~\ref{charsolv}. In particular, the associative subalgebra generated
by $x$ is nilpotent. Therefore, $\mathcal{P}_A$ is a nilalgebra. In finite dimension, a nil commutative
associative algebra is nilpotent, which proves the last assertion.
\end{proof}

\begin{remark}
The finite-dimensional hypothesis on $U$ in Theorem~\ref{nilcharacterization} cannot be
removed. Let 
$$\mathcal{P}=(x_1,x_2,\ldots)\subseteq
    \mathbb{F}[x_1,x_2,\ldots]/(x_1^2,x_2^2,\ldots),$$
and endow $\mathcal{P}$ with the zero Lie bracket. Then $\mathcal{P}_A$ is nil but not nilpotent.
Now $\mathcal{P}_L$ is abelian, hence nilpotent. Taking
$U=\mathcal{P}_L$, we have $S_U=\mathcal{P}$, which is not nilpotent.
Thus condition $(i)$ does not imply condition $(ii)$ for arbitrary
infinite-dimensional Lie subalgebras $U$.
\end{remark}

\section{Poisson closures of some Lie subalgebras}\label{sec4}

We now consider the complementary question of when the Poisson closure does not enlarge the
original Lie subalgebra. In this section we study the Poisson closures of some canonical Lie subalgebras of
$\mathcal{P}_L$. A Lie subalgebra $U$ of $\mathcal{P}_L$ will be called {\it stable}
under Poisson closure if $S_U=U$. Equivalently, $U$ is stable when it
is a Poisson subalgebra of $\mathcal{P}$. Otherwise, $U$ will be called {\it unstable} under Poisson closure.

\subsection{Stability of Lie subalgebras} The following proposition collects the main examples of stable Lie subalgebras. For every $a\in\mathcal P$, recall that the Engel subalgebra associated with $a$ is 
$$
E_a:=\{x\in\mathcal P:Q_a^n(x)=0\text{ for some }n=n(x)\geq 1\}.
$$

\begin{propo}
    Let $\mathcal{P}$ be a Poisson algebra over a field $\mathbb{F}$. Then the following subalgebras of  $\mathcal{P}_L$ are stable:
    \begin{enumerate}
        \item the centralizer of any set $X\subseteq \mathcal{P}$. Moreover, the center of $\mathcal{P}_L$ is stable;
        \item the maximal abelian subalgebras of $\mathcal{P}_L$ are stable; 
        \item Engel subalgebras are stable;
        \item if $\mathcal{P}$ has dimension $n$ and $|\mathbb{F}|\geq n$, then any Cartan subalgebra of $\mathcal{P}_L$ is stable; 
        \item if $\mathcal{P}$ is finite-dimensional and solvable, then any Cartan subalgebra of $\mathcal{P}_L$ is stable.
    \end{enumerate}
\end{propo}
\begin{proof} Let $\mathcal{P}$ be a Poisson algebra over a field $\mathbb{F}$. We prove each claim separately.
    \begin{enumerate}
        \item Let $X$ be a subset of $\mathcal{P}$. Then the centralizer $C_{\mathcal{P}_L}(X)$ is a subalgebra of $\mathcal{P}_L$. Also, for every $x,y\in C_{\mathcal{P}_L}(X)$ we have 
        $$[xy, X] = y[x, X] + x[y, X] = 0.$$ So $C_{\mathcal{P}_L}(X)$ is a subalgebra of $\mathcal{P}$.
        In particular, the center of $\mathcal{P}_L$ is $C_{\mathcal{P}_L}(\mathcal{P}_L)$. 
        \item Let $U$ be a maximal abelian subalgebra of $\mathcal{P}_L$. Then $U = C_{\mathcal{P}_L}(U)$ and $U$ is stable by the previous paragraph.
        
        \item 
    Given $a\in\mathcal{P}$, we show that for every $x, y \in E_a$, we obtain $xy\in E_a$. Indeed, we have 
    $$Q_a^k(xy) = \sum_{i=0}^{k}\binom{k}{i} Q^i_a(x)Q_a^{k-i}(y).$$
    If $k> r + s$, for $r,s$ with $Q^r_a(x) = 0$ and $Q^s_a(y) = 0$, we have $Q_a^k(xy) = 0$, so $xy\in E_a$.  Therefore, any Engel subalgebra $E_a$ is stable.

    \item Let $C$ be a Cartan subalgebra of $\mathcal{P}_L$. If the field has at least $n$ elements, then there is a regular element $a\in \mathcal{P}$ such that $C = E_a$, by \cite[Theorem 1]{barnes} ($C$ is a minimal Engel in Barnes' terminology).  Hence, $C$ is stable.
    
    \item Assume that $\mathcal{P}$ is solvable. By Proposition~\ref{charsolv},
$\mathcal{P}_A$ is nilpotent and $\mathcal{P}_L$ is solvable. Let $C$ be a
Cartan subalgebra of $\mathcal{P}_L$. Since $C$ is nilpotent and
self-normalising, it is a maximal nilpotent subalgebra of $\mathcal{P}_L$. By
Proposition~\ref{lie2}, $C$ is a subalgebra of $\mathcal{P}$.
\end{enumerate}    
This proves all the claims.
\end{proof}

Under additional assumptions, maximal nilpotent and maximal solvable Lie
subalgebras are also stable under Poisson closure as a consequence of results of the previous section. 

\begin{propo}\label{lie2} \label{lie2s}
    Let $\mathcal{P}$ be a Poisson algebra. Suppose $\mathcal{P}_A$ is nil.
\begin{enumerate}
    \item If $U$ is a finite-dimensional maximal nilpotent subalgebra of
    $\mathcal{P}_L$, then $U$ is stable.

    \item If $U$ is a finite-dimensional maximal solvable subalgebra of
    $\mathcal{P}_L$, then $U$ is stable.
\end{enumerate}
\end{propo}
\begin{proof}
    Let $U$ be a finite-dimensional maximal nilpotent subalgebra of
$\mathcal{P}_L$. Since $\mathcal{P}_A$ is a nilalgebra, Theorem~\ref{nilcharacterization}
implies that $S_U$ is a nilpotent Poisson subalgebra of $\mathcal{P}$. Hence
$(S_U)_L$ is a nilpotent Lie subalgebra of $\mathcal{P}_L$ containing $U$.
By the maximality of $U$, we obtain
$S_U=U$. 

An analogous argument using Theorem~\ref{nilcharacterization} proves the second claim.
\end{proof}

\begin{propo}
    Let $\mathcal{P}$ be a Poisson algebra. Suppose $\mathcal{P}_A$ is nilpotent.
\begin{enumerate}
    \item If $U$ is a maximal nilpotent subalgebra of
    $\mathcal{P}_L$, then $U$ is stable.

    \item If $U$ is a maximal solvable subalgebra of
    $\mathcal{P}_L$, then $U$ is stable.
\end{enumerate}
\end{propo}
\begin{proof}
    Let $U$ be a maximal nilpotent subalgebra of $\mathcal{P}_L$. Then we have that $S_U$ is a nilpotent subalgebra and in particular $(S_U)_L$ is a nilpotent subalgebra of  $\mathcal{P}_L$, by Proposition~\ref{propnil}. Hence, the maximality of $U$ implies that $S_U = U$. 
    
    An analogous argument using Proposition~\ref{propsolv} proves the second claim.
\end{proof}

\begin{remark}\label{remex2}
The assumption that $\mathcal{P}_A$ is nil cannot be omitted in
the previous propositions.
    Assume $\mathbb{F}$ has characteristic $p\geq 2$. Consider the finite-dimensional algebra $\mathbb{F}[x, y]/(x^p, y^p)$ with the bracket as in (\ref{eqbr}). This quotient is well-defined since $(x^p, y^p)$ is in the center of the Lie bracket. Choose $U$ the span of $\left\{1, y, x^k: 1\leq k\leq p-1\right\}$. Then $U$ is a maximal nilpotent Lie subalgebra. But $U$ is not a subalgebra of $\mathbb{F}[x, y]/(x^p, y^p)$, because $xy\not\in U$. 
\end{remark}

\subsection{Poisson closure of the derived algebra}  In this section, we write $D := \mathcal{P}_L^2 = [\mathcal{P}, \mathcal{P}]$. 
Following the example discussed in Remark~\ref{remex2}, one can see that for $A= \mathbb{F}[x, y]/(x^p, y^p)$, where $\mathbb{F}$ is a field of characteristic $p\geq 2$, the derived algebra $D:= A_L^{2}$ is not stable under the Poisson closure since $x^{p-1}y^{p-1}\in S_D\setminus D$. One can easily check this using that 
$$[x^a y^b, x^c y^d]  = (ad-bc) x^{a+c-1} y^{b+d-1}.$$

Likewise, for the Laurent polynomials in two variables $B=\mathbb{C}[x^{\pm}, y^{\pm}]$ with its standard Lie bracket (\ref{eqbr}), we have $x^{-1}y^{-1}\in S_D\setminus D$, where $D = B_L^{2}$.

The previous examples show that the derived algebra is not stable in general, which makes its Poisson closure a genuinely nontrivial object to consider.

\begin{propo}\label{wsu}
    Let $\mathcal{P}$ be a Poisson algebra over a field of characteristic $p$. 
    \begin{enumerate}
        \item If $\mathcal{P}_L$ is solvable and $p\neq 2$, then $(S_{D})_A$ is nil.

        \item If $\mathcal{P}_L$ is solvable and $p\geq 3$, then $(S_{D})_A$ is nil of bounded index.

        \item If $\mathcal{P}_L$ is nilpotent and $p>0$, then $(S_{D})_A$ is nil of bounded index.
    \end{enumerate}   
\end{propo}
\begin{proof}
    Assume $\mathcal{P}_L$ is solvable.  Then $\mathcal{P} [\mathcal{P}, \mathcal{P}]$ is a nil ideal of $\mathcal{P}_A$, by \cite[Theorem 3.6 1)]{SU21b}. 
    Notice that $S_{D}\subseteq [\mathcal{P}, \mathcal{P}] + \mathcal{P}[\mathcal{P}, \mathcal{P}]$. Now, if $\mathcal{P}$ has a unit, we have $[\mathcal{P}, \mathcal{P}]  \subseteq \mathcal{P}[\mathcal{P}, \mathcal{P}]$.  Otherwise, we can consider the unital Poisson algebra $\mathcal{P}'=\mathcal{P}+ \mathbb{F}1$. This algebra is Lie solvable, so $\mathcal{P}'[\mathcal{P}', \mathcal{P}']$ is associative nil. Then it follows that $$S_{D} = S_{[\mathcal{P}', \mathcal{P}']}\subseteq [\mathcal{P}', \mathcal{P}'] + \mathcal{P}'[\mathcal{P}', \mathcal{P}']\subseteq \mathcal{P}'[\mathcal{P}', \mathcal{P}'].$$
    Therefore, we conclude that $(S_{D})_A$ is nil.

    Analogously, we can prove the other statements using \cite[Theorem 3.6]{SU21b}. 
\end{proof}

Proposition~\ref{wsu} naturally raises the question of 
whether, for a Lie-solvable Poisson algebra, the associative algebra
$(S_D)_A$ is not only nil but actually nilpotent. We show that this is
not the case, even over a field of characteristic zero. For this purpose,
we recall the Poisson--Grassmann algebra introduced in
\cite[Section 6]{MP07}.

Let $\mathbb F$ be a field of characteristic $p\neq2$ and let $V= \textrm{span}( e_1,e_2,\ldots)$. Let $\Lambda(V)$ be the unital Grassmann algebra over $V$, whose multiplication will be denoted by $\wedge$. For $a,b\in\Lambda(V)$, define
$$
a\cdot b=\frac{1}{2}(a\wedge b+b\wedge a)
\qquad\textrm{and}\qquad
[a,b]=a\wedge b-b\wedge a.
$$
Then $\mathbf G=(\Lambda(V),\cdot, [\cdot, \cdot])$ is a Poisson algebra, called the \emph{Poisson--Grassmann algebra}.

\begin{propo}
$\mathbf G_L^2\neq0$ and $\mathbf G_L^3=0$. Moreover, we have
$$
\mathbf G_L^2=
\textrm{span}( e_{i_1}\wedge\cdots\wedge e_{i_{2n}}:
i_1<\cdots<i_{2n},\ n\geq1).
$$
Furthermore, $(\mathbf G_L^2)_A$ is nil but not nilpotent.
\end{propo}
\begin{proof}
    If $a,b$ are homogeneous elements of $\Lambda(V)$, then $a\cdot b=0$ and $[a,b]=2a\wedge b$ if both $a$ and $b$ have odd degree, whereas $a\cdot b=a\wedge b$ and $[a,b]=0$ otherwise. It follows that $\mathbf G_L^2$ is precisely the span of the Grassmann monomials of positive even degree. In particular, we have $[\mathbf G_L^2,\mathbf G]=0$, so $\mathbf G_L^3=0$, whereas $[e_1,e_2]=2e_1\wedge e_2\neq0$, so $\mathbf G_L^2\neq 0$.

    Every element of $\mathbf G_L^2$ is associatively nilpotent. However, $(\mathbf G_L^2)_A^n\neq0$ for every $n\geq1$, as
    $$(e_1\wedge e_2)\cdot(e_3\wedge e_4)\cdots(e_{2n-1}\wedge e_{2n})
    =e_1\wedge\cdots\wedge e_{2n}\neq0.$$
    Thus $(\mathbf G_L^2)_A$ is nil but not nilpotent.
\end{proof}

Now, consider the tensor product $\mathbf T=\mathbf G\otimes\mathbf G$. Recall that the tensor product of two Poisson algebras is again a Poisson algebra, with multiplications given, for $a,b,c,d\in\mathbf G$, by
$$
(a\otimes b)\cdot(c\otimes d)=(a\cdot c)\otimes(b\cdot d), \qquad
[a\otimes b,c\otimes d]=[a,c]\otimes(b\cdot d)+(a\cdot c)\otimes[b,d].
$$

\begin{propo}
The Poisson algebra $\mathbf T$ satisfies the following properties:
\begin{enumerate}
\item $\mathbf T_L$ is solvable non-nilpotent. More precisely,
$\mathbf T_L^{(2)}\neq 0$, $\mathbf T_L^{(3)}=0$, and
$\mathbf T_L^n\neq 0$ for $n\geq 1$;
\item $(S_{\mathbf T_L^2})_A$ is nil but not nilpotent;
\item the Poisson ideal generated by $\mathbf T_L^{(2)}$ is not nilpotent as an associative algebra.
\end{enumerate}
\end{propo}
\begin{proof}
    Since $\mathbf G$ has an identity element, we have
$\mathbf T_L^2=\mathbf G_L^2\otimes\mathbf G+\mathbf G\otimes\mathbf G_L^2$.
Using $[\mathbf G_L^2,\mathbf G]=0$, it follows that
$$
\mathbf T_L^{(2)}
=(\mathbf G_L^2)_A^2\otimes\mathbf G_L^2+
\mathbf G_L^2\otimes(\mathbf G_L^2)_A^2.
$$
Hence, we have $\mathbf T_L^{(2)}\neq0$ and $\mathbf T_L^{(3)}=0$, because we have $[(\mathbf G_L^2)_A^2, \mathbf G]=0$ by Leibniz rule.

On the other hand, $\mathbf T_L^n\neq0$ for every $n$. Indeed, put $u_2=[e_1\otimes e_1,e_2\otimes1]$ and, recursively, $u_n=[u_{n-1},e_n\otimes e_{n-1}]$ for $n\geq3$. A straightforward induction yields
$$
u_n=2^{n-1}(e_1\wedge\cdots\wedge e_n)\otimes
(e_1\wedge\cdots\wedge e_{n-1})\neq0.
$$
Since $u_n\in\mathbf T_L^n$, this proves $i)$.

By Proposition~\ref{wsu}, $(S_{\mathbf T_L^2})_A$ is nil. However,
$(\mathbf G_L^2)_A^2\otimes\mathbf G_L^2\subseteq
\mathbf T_L^{(2)}\subseteq S_{\mathbf T_L^2}$ and, for $n\geq1$,
$$
\left((\mathbf G_L^2)_A^2\otimes\mathbf G_L^2\right)_A^n
=(\mathbf G_L^2)_A^{2n}\otimes(\mathbf G_L^2)_A^n\neq0.
$$

Thus $(S_{\mathbf T_L^2})_A$ is not nilpotent, proving $ii)$. Since the Poisson ideal generated by $\mathbf T_L^{(2)}$ contains $(\mathbf G_L^2)_A^2\otimes\mathbf G_L^2$, it is not associative nilpotent either. This proves $iii)$.
\end{proof}

Thus the conclusion of Proposition~\ref{wsu} cannot in general be strengthened from associative nil to associative nilpotent. The same example also shows that the last Lie bracket in \cite[Theorem 3.3]{SU21b} is needed: although the Poisson ideal generated by $[\mathbf T_L^{(2)},\mathbf T]$ is associative nilpotent, the Poisson ideal generated by $\mathbf T_L^{(2)}$ is not.

\section{Poisson $A$-algebras}
\label{sec5}

We now apply the previous results to describe a distinguished class of Poisson
algebras: those whose nilpotent subalgebras are zero algebras. We first establish some basic properties and then obtain a structural
description in the finite-dimensional case. We conclude with results on solvable Poisson $A$-algebras.

\subsection{Basic properties} We begin by introducing Poisson $A$-algebras and recording some elementary properties that will be used throughout the section.

\begin{definition} We say that a Poisson algebra $\mathcal{P}$ is an $A$-algebra if every nilpotent subalgebra of it is a zero algebra. Equivalently, for every nilpotent subalgebra $U$ of $\mathcal{P}$, we have $U^2_A + [U, U] = 0$. 
\end{definition}

\begin{lemma}\label{l:fac}If $\mathcal{P}$ is a finite-dimensional Poisson $A$-algebra over any field and $B$ is an ideal of $\mathcal{P}$, then $\mathcal{P}/B$ is a Poisson $A$-algebra
\end{lemma}
\begin{proof} Let $U$ be a subalgebra of $\mathcal{P}$ such that $U/B$ is nilpotent. If $B\subseteq \phi(U)$ then $U$ is nilpotent, by \cite[Theorem 4.1]{nsfp}, and hence a zero algebra. 
\par

So suppose that $B \not \subseteq \phi(U)$. Then there is a maximal subalgebra $M$ of $U$ such that $U=B+M$. Choose $C$ to be a subalgebra of $U$ which is minimal with respect to $U=B+C$. Then $B\cap C\subseteq \phi(C)$ and $U/B \cong C/(B\cap C)$, by \cite[Lemma 3.5]{nsfp}. It follows, by \cite{nsfp} again, that $C$ is nilpotent and hence a zero algebra.
\par

Hence, in either case, $U/B$ is a zero algebra and $\mathcal{P}/B$ is an $A$-algebra.  
\end{proof}

\begin{lemma}\label{l:lemm2} Let $B$, $C$ be ideals of the Poisson algebra $\mathcal{P}$.
\begin{itemize}
\item[(i)] If $\mathcal{P}/B$, $\mathcal{P}/C$ are $A$-algebras, then $\mathcal{P}/(B \cap C)$ is an $A$-algebra.
\item[(ii)] If $\mathcal{P} = B \oplus C$, where $B, C$ are $A$-algebras, then $\mathcal{P}$ is an $A$-algebra.
\end{itemize}
\end{lemma}
\begin{proof} (i) Let $U/(B \cap C)$ be a nilpotent subalgebra of $\mathcal{P}/(B \cap C)$. Then $(U + B)/B$ is a nilpotent subalgebra of $\mathcal{P}/B$, which is an $A$-algebra. It follows that $U^2 \subseteq B$. Similarly, $U^2 \subseteq C$, whence the result.
\par
(ii) This follows from (i).
\end{proof}

\begin{lemma}\label{cor33}
    If $\mathcal{P}$ is a Poisson $A$-algebra over a field of characteristic $p\neq 2$ in which $\mathcal{P}_A$ is nil, then $\mathcal{P}_A^2 = 0$ and $\mathcal{P}$ is just a Lie algebra. Moreover, $\mathcal{P}_L$ is a Lie $A$-algebra.
\end{lemma}
\begin{proof}
    Given $x\in \mathcal{P}$, consider $U$ the Lie subalgebra $\mathbb{F}x$. Then $S_U$ is a nilpotent subalgebra of $\mathcal{P}$, by Theorem~\ref{nilcharacterization}. Since $\mathcal{P}$ is an $A$-algebra, $S_U$ is a zero algebra, so we have $x^2=0$, in particular. From here, for any $x, y \in \mathcal{P}$ we obtain $0=(x+y)^2= x^2+y^2+ 2xy = 2xy.$
    Hence, we conclude $xy=0$ and $\mathcal{P}_A^2=0$.

    Finally, let $V$ be a nilpotent Lie subalgebra of $\mathcal P_L$. Since the associative product on $\mathcal P$ is zero, $V$ is a nilpotent Poisson subalgebra of $\mathcal P$. As $\mathcal P$ is a Poisson $A$-algebra, $V$ is a zero algebra, and therefore $[V,V]=0$. Thus every nilpotent subalgebra of $\mathcal P_L$ is abelian, so $\mathcal P_L$ is a Lie $A$-algebra.
\end{proof}

The above result is false over a field of characteristic $2$, as the following example shows.

\begin{remark}
Let $\mathcal{P}$ have basis $x,y,z$ where $[x,y]=y$, $[x,z]=z$, $xy=z$. It is straightforward to check that this is a Poisson algebra over a field of characteristic $2$. The two-dimensional subalgebras of $\mathcal{P}$ are $\mathbb{F}z+\mathbb{F}(\lambda x+\mu y)$ for any $\lambda,\mu\in \mathbb{F}$ and the only one of these that is nilpotent is $\mathbb{F}z+\mathbb{F}y$, which is a zero subalgebra. It follows that $\mathcal{P}$ is an $A$-algebra. However, $\mathcal{P}_A$ is nilpotent but not a zero algebra.
\end{remark}

{Regardless of the characteristic, the following result holds.

\begin{propo}
    If $\mathcal{P}$ is a Poisson $A$-algebra with $\mathcal{P}_A$ nilpotent, then $\mathcal{P}_A^4 = 0$ and $(\mathcal{P}_A^2)_L^2 = 0$. Moreover, if $\mathcal{P}$ is finite-dimensional, then $\mathcal{P}_A^3 = 0$.
\end{propo}
\begin{proof}
    Notice that $B = \mathcal{P}_A^2$ is a nilpotent subalgebra of $\mathcal{P}_L$, because $[\mathcal{P}_A^i, \mathcal{P}_A^j]\subseteq \mathcal{P}_A^{i+j-1}$ and $\mathcal{P}_A$ is nilpotent. So $S_B$ is a Poisson nilpotent subalgebra of $\mathcal{P}$. Hence $B$ is a zero algebra as stated.
    
    Now, we prove the second part. The argument in Lemma~\ref{cor33} still gives $x^2=0$ for all $x\in \mathcal{P}$. Suppose $\mathcal{P}_A^3 \neq 0$ and choose $a\in \mathcal{P}$ such that $P_a|_B\neq 0$. By the previous paragraph it follows $a\not\in B$.  Assume $Q_a$ acts invertibly on $B$. Then we have $P_a(B) = 0$. Indeed, we have 
    $$P_aQ_a(xy) = a[a, xy] = [a, axy] = [ax, ay] +[a^2 y, x] = [ax, ay] \in [B,B] = 0$$
    for $x, y \in \mathcal{P}$, so $P_a(B) = P_aQ_a(B) = 0$. 
    Hence $Q_a|_B$ is not invertible. 

    Since $P_a|_B\neq 0$, we can choose $b\in B$ such that $P_a(b)\neq 0$. Since we have $Q_a(B)\subseteq \textrm{ker}(P_a|_B)$, then $b\not\in Q_a(B)$, so $b$ is not in the $1$-space of the Fitting decomposition of $Q_a|_B$. Here we use the finite-dimensionality of $\mathcal{P}$. Take $b_0\neq 0$ the projection of $b$ on the nullspace of $Q_a|_B$. Then $Q_a^k(b_0) =0$ for some $k\in \mathbb{N}$ and $P_a(b_0)\neq 0$. 

    Finally, consider the Lie subalgebra $U$ generated by $a$ and $b_0$. Then $U$ is nilpotent, because $U_L^{k+1}\subseteq Q_a^{k}(U) = 0$, so $S_U$ is a Poisson nilpotent subalgebra of $\mathcal{P}$, by Theorem~\ref{nilcharacterization}. Since $\mathcal{P}$ is a Poisson $A$-algebra, $S_U$ must be a zero algebra, but $ab_0\neq 0$, which is a contradiction. Therefore, we conclude $\mathcal{P}_A^3=0$.
\end{proof}}

\begin{remark}
The question whether $\mathcal{P}_A^3=0$ holds without finite-dimensionality
is only open in characteristic $2$. In characteristic $p\neq 2$, one already has $\mathcal{P}_A^2=0$. The finite-dimensional hypothesis in the last proposition is used on the
Fitting decomposition of the linear map
$Q_a|_{\mathcal{P}_A^2}$. 
\end{remark}

\subsection{Structure of finite-dimensional Poisson $A$-algebras}
We now turn to the general finite-dimensional case, where idempotents and the associative radical play a central role in the structure.

We recall the Peirce decomposition for Poisson algebras given in \cite[Proposition 10]{GR08}, which will be used in our description of finite-dimensional Poisson $A$-algebras.

\begin{lemma}\label{idem} If $e$ is an idempotent of the associative algebra $\mathcal P_A$, then $[\mathcal P,e]=0$.
\end{lemma}
\begin{proof} For all $x\in \mathcal{P}$, $[x,e]=[x,e\cdot e]=2e\cdot[x,e]=4e[x,e]$, so $[x,e]=0$ irrespective of the characteristic.
\end{proof}

\begin{lemma}\label{peirce} Assume $\mathcal{P}$ is finite-dimensional. If $\mathcal{P}_A$ is not nilpotent, then $\mathcal{P}_A=(1-e)\cdot\mathcal{P}\oplus e\cdot\mathcal{P}$ where $e$ is a principal idempotent and $(1-e)\cdot \mathcal{P}$ and $e\cdot \mathcal{P}$ are ideals of $\mathcal{P}$.
\end{lemma}
\begin{proof} Suppose that $\mathcal{P}_A$ is not nilpotent. Then it is well known that there is a principal idempotent and the Peirce decomposition is $\mathcal{P}_A=(1-e)\cdot\mathcal{P}+e\cdot\mathcal{P}$. Now, for all $x,y\in \mathcal{P}$, $x\cdot (y-e\cdot y)=(1-e)\cdot x\cdot y\in (1-e)\cdot \mathcal{P}$ and $[x,y-e\cdot y]=[x,y]-[x,e\cdot y]=[x,y]-e\cdot[x,y]-y\cdot[x,e]=(1-e)\cdot[x,y]\in (1-e)\cdot \mathcal{P}$, by Lemma~\ref{idem}. Also,
$x\cdot e\cdot y=e \cdot x\cdot y\in e\cdot \mathcal{P}$,  and $[x,e\cdot y]=e\cdot [x,y]\in e\cdot \mathcal{P}$, using Lemma~\ref{idem} again. {Finally, if $u=e\cdot x=y-e\cdot y\in e\cdot \mathcal{P}\cap (1-e)\mathcal{P}$, then $e\cdot u=e\cdot x=0$ so $u=0$.}
\end{proof}

\begin{propo} Let $\mathcal P$ be a finite-dimensional Poisson $A$-algebra over a field $\mathbb F$ of characteristic different from $2$ and $3$ and of cohomological dimension at most $1$. Then
$$
   Z(\mathcal P)\cap\mathcal P^{(1)}=0.
$$
\end{propo}
\begin{proof} 
If $\mathcal{P}_A$ is nilpotent, then Lemma~\ref{cor33} gives
$\mathcal{P}_A^2=0$ and $\mathcal{P}_L$ is a Lie $A$-algebra. Hence
$Z(\mathcal P)=Z(\mathcal P_L)$ and
$\mathcal P^{(1)}=\mathcal P_L^2$, so the result follows from
\cite[Proposition~2]{PS90}. We may therefore assume that
$\mathcal{P}_A$ is not nilpotent. Let $e\in \mathcal{P}_A$ be a principal idempotent. The associative algebra $(1-e)\mathcal{P}$ is nilpotent. Hence, $((1-e)\mathcal{P})_A^2 = 0$ and $(1-e)\mathcal{P}$ is a Lie $A$-algebra, by Lemma~\ref{cor33}. Moreover, we have $(e\mathcal{P})_A^2 = e\mathcal{P}$, so $\mathcal{P}^{(1)} = ((1-e)\mathcal{P})_L^2 \oplus e\mathcal{P}$. Write $I = (1-e)\mathcal{P}$.

Now, let $x\in Z(\mathcal P)\cap\mathcal P^{(1)}$. Since $x\mathcal{P} = 0$, we have $ex=0$ and hence $x\in I_L^2$. Therefore, we have $x\in Z(I_L)\cap I_L^2$. Since $I_L$ is a finite-dimensional Lie $A$-algebra, we conclude $x=0$, by \cite[Proposition~2]{PS90}.
\end{proof}

\begin{propo}\label{jacobideal}
    Let $\mathcal{P}$ be a finite-dimensional Poisson $A$-algebra over a field of characteristic different from $2$. Then $J(\mathcal{P}_A)$, the radical of $\mathcal{P}_A$, is an ideal of $\mathcal{P}$.
\end{propo}
\begin{proof}
    Let $x \in J(\mathcal{P}_A)$ and $y\in \mathcal{P}$. Since $\mathcal{P}_A$ is finite-dimensional, $x$ is associatively nilpotent. Hence the Poisson subalgebra $S_{\mathbb{F}x}$ is nilpotent and therefore, since $\mathcal{P}$ is an $A$-algebra, it is a zero algebra. In particular, since $x^2=0$, we have $2x[x,y]=0$. Then
    $$[x,y]^3=[x[x,y]^2,y]-x[[x,y]^2,y]=-2x[x,y][[x,y],y]=0.$$ 
    Therefore, $[x,y]$ is associatively nilpotent and hence belongs to $J(\mathcal{P}_A)$.
\end{proof}

\begin{theor}\label{quotient}
    Let $\mathcal{P}$ be a finite-dimensional Poisson $A$-algebra over a field $\mathbb{F}$  of characteristic different from $2$. Then as Poisson algebras
    $$\mathcal{P} / J(\mathcal{P}_A) \cong \mathbb{K}_1 \oplus \ldots \oplus \mathbb{K}_m,$$ 
    where $\mathbb{K}_1, \ldots, \mathbb{K}_m$ are finite field extensions of $\mathbb{F}$. 
\end{theor}
\begin{proof}
    Let $\mathcal{P}$ be a finite-dimensional Poisson $A$-algebra over a field $\mathbb{F}$. By the Wedderburn theorem, we obtain $\mathcal{P}_A/ J(\mathcal{P}_A) \cong  \mathbb{K}_1 \oplus \ldots \oplus \mathbb{K}_m$, where $\mathbb{K}_i$ are finite field extensions of $\mathbb{F}$. Since $J(\mathcal{P}_A)$  is an ideal of $\mathcal{P}$, the quotient $\mathcal{P}/ J(\mathcal{P}_A)$ is well-defined.
    
    It remains to prove that $[\mathbb{K}_i, \mathbb{K}_j] = 0$ when $i\neq j$ and $[\mathbb{K}_i, \mathbb{K}_i] \subseteq \mathbb{K}_i$. Assume $e_i$ is the unit of $\mathbb{K}_i$. Since for $x\in \mathbb{K}_i$ and $y\in \mathbb{K}_j$, we have
    $$[x, y] = [e_i x, y] = e_i [x, y] + x[e_i, y] = e_i [x, y],$$
    and $e_i\mathbb{K}_j = 0$ for $i\neq j$, it follows that $[x, y]\in \mathbb{K}_i$. Similarly, we obtain $[x, y]\in \mathbb{K}_j$. Hence, we deduce $[x,y] = 0$ when $i\neq j$. An analogous argument gives $[\mathbb{K}_i, \mathbb{K}_i] \subseteq \mathbb{K}_i$. 
\end{proof}

\begin{coro}
    If $\mathcal{P}$ is a finite-dimensional Lie solvable Poisson $A$-algebra over a field of characteristic different from $2$, then $\mathcal{P}_L^2\subseteq J(\mathcal{P}_A)$.
\end{coro}
\begin{proof}
    Write $C = \mathcal{P} / J(\mathcal{P}_A)$. By Theorem~\ref{quotient}, we have $C \cong \mathbb{K}_1 \oplus \ldots \oplus \mathbb{K}_m,$ for certain finite field extensions of $\mathbb{F}$. Since $C_L$ is solvable, Proposition~\ref{wsu} shows that $(S_{C_L^2})_A$ is nil. Since $C$ is finite-dimensional, $(S_{C_L^2})_A$ is therefore nilpotent. But $C_A$ is a direct sum of fields, so it has no nilpotent subalgebras.  Hence, $S_{C_L^2} = 0$ and $C_{L}^2 =0$. It follows  that $\mathcal{P}_L^2 \subseteq J(\mathcal{P}_A)$.
\end{proof}

\begin{remark}
    Consider the field $\mathbb{F} = \mathbb{F}_3(r^3, s^3)$ and its extension $\mathbb{K} = \mathbb{F}_3(r, s)$, for algebraically independent indeterminates $r,s$. Then $\mathbb{K}$ is a $9$-dimensional algebra over $\mathbb{F}$ with a basis $$\left\{1, r, r^2, s, s^2, rs, r^2s, rs^2, r^2s^2\right\}.$$ Consider the Poisson algebra $\mathcal{P}$ obtained by endowing $\mathbb{K}$ with the usual symplectic bracket
     \begin{equation}\label{eqbr2}
        [f, g] = \dfrac{\partial f}{\partial r}\dfrac{\partial g}{\partial s} - \dfrac{\partial g}{\partial r}\dfrac{\partial f}{\partial s}.
    \end{equation} 
    Notice that $ \frac{\partial }{\partial r}$ and $\frac{\partial }{\partial s}$ are commuting derivations of $\mathbb{K}$.  
    Then $\mathbb{K}$ does not have nilpotent elements, so $\mathcal{P}$ is an $A$-algebra. Also, the Lie bracket is not abelian.
\end{remark}

Recall that if $\mathbb{K}$ is a finite separable extension of $\mathbb F$, then every $\mathbb F$-derivation from $\mathbb K$ into a $\mathbb K$-module is zero. In fact,  in the previous example $\mathbb{K}$ is not separable as the minimal polynomial of $r$ is $x^3-r^3 = (x-r)^3$, which allows a non abelian Lie bracket. 

In the separable case, the situation is different. 
Let $\mathcal{A}$ and $\mathcal{B}$ be two Poisson algebras. We write $\mathcal{P} = \mathcal{A} \ltimes \mathcal{B}$ if $\mathcal{P}$ is a Poisson algebra such that $\mathcal{A}$ is a subalgebra of $\mathcal{P}$, $\mathcal{B}$ is an ideal of $\mathcal{P}$ and, as vector spaces, $\mathcal{P} = \mathcal{A} \oplus \mathcal{B}$.

\begin{coro}\label{corosep}
    Let $\mathcal{P}$ be a finite-dimensional Poisson $A$-algebra over a field $\mathbb{F}$ of characteristic $p\neq 2$. If $\mathcal{P} / J(\mathcal{P}_A)$ is a direct sum of separable field extensions, then as Poisson algebras
    $$\mathcal{P} = (\mathbb{K}_1 \oplus \ldots \oplus \mathbb{K}_m) \ltimes J(\mathcal{P}_A),$$ 
    where $\mathbb{K}_1, \ldots, \mathbb{K}_m$ are finite separable field extensions of $\mathbb{F}$ and 
    \begin{enumerate} 
        \item 
    $J(\mathcal{P}_A)_A^2 = 0$ and $J(\mathcal{P}_A)_L$ is a Lie $A$-algebra; 
        \item $[\mathcal{P}, \mathbb{K}_1 \oplus \ldots \oplus \mathbb{K}_m]=0$.
    \end{enumerate}
\end{coro}
\begin{proof}
    The result follows by a similar argument to that of Theorem~\ref{quotient}, using the Malcev lift of Wedderburn theorem and the fact that separable field extensions do not admit non trivial derivations into any $\mathbb{K}_i$-module, in particular into $\mathcal{P}$.    
\end{proof}

{The preceding decomposition can be refined further. We first introduce
the building blocks that will appear in this description. Let $\mathcal L$ be a Lie algebra over a field $\mathbb K$. We denote by
$\widehat{\mathcal L}^{\,\mathbb K}$
the \emph{unitization} of $\mathcal L$, which consists of the vector space $\mathbb K\oplus\mathcal L$ endowed with the multiplications
$$
(\lambda,x)(\mu,y)
=
(\lambda\mu,\lambda y+\mu x),
\qquad
[(\lambda,x),(\mu,y)]
=
(0,[x,y]),
$$
for $\lambda,\mu\in\mathbb K$ and $x,y\in\mathcal L$. Thus
$\widehat{\mathcal L}^{\,\mathbb K}$ is a Poisson algebra over
$\mathbb K$ in which $\mathcal L$ is an ideal with zero associative
multiplication. If $\mathbb K/\mathbb F$ is a field extension, we
will regard $\widehat{\mathcal L}^{\,\mathbb K}$ as a Poisson
algebra over $\mathbb F$ by restriction of scalars.

\begin{propo}
If $\mathcal L$ is a Lie $A$-algebra over $\mathbb K$, then
$\widehat{\mathcal L}^{\,\mathbb K}$ is a Poisson $A$-algebra.
\end{propo}
\begin{proof}
Let $U$ be a nilpotent subalgebra of
$\widehat{\mathcal L}^{\,\mathbb K}$. Then $U\subseteq\mathcal L$, since
$\mathbb K$ is a field and hence contains no nonzero nilpotent elements.
So $U$ is a zero algebra, because $\mathcal L_A^2=0$ and
$\mathcal L$ is a Lie $A$-algebra.
\end{proof}

\begin{remark}
    A standard example of a Lie $A$-algebra is
    $\mathfrak{aff}_1(\mathbb{C})$. Hence, the previous proposition
    provides a nontrivial example of a Poisson $A$-algebra. Moreover, regarding this algebra over $\mathbb{R}$ by
    restriction of scalars, we obtain a Poisson $A$-algebra over
    $\mathbb{R}$. This shows that, in a structure theorem over a general
    base field, one must allow unitizations over proper field extensions.
\end{remark}

We will show that unitizations of this form are the natural building
blocks of finite-dimensional Poisson $A$-algebras under the
separability assumption appearing above.

\begin{theor}\label{refined}
Let $\mathcal P$ be a finite-dimensional Poisson $A$-algebra over a
field $\mathbb F$ of characteristic different from $2$. If
$\mathcal P/J(\mathcal P_A)$ is separable over $\mathbb F$, then $\mathcal P$ can be written as a direct sum of Poisson ideals
$$
\mathcal P
\cong
\mathcal L_0\oplus
\widehat{\mathcal L}_1^{\,\mathbb K_1}
\oplus\cdots\oplus
\widehat{\mathcal L}_m^{\,\mathbb K_m}
$$
where
$\mathbb K_1,\ldots,\mathbb K_m$ are finite separable field
extensions of $\mathbb F$, $\mathcal L_0$ is a Lie $A$-algebra over
$\mathbb F$, and $\mathcal L_i$ is a Lie $A$-algebra over
$\mathbb K_i$ for $1\leq i\leq m$.
\end{theor}
\begin{proof}
    By Corollary~\ref{corosep}, we can write $\mathcal{P} = (\mathbb{K}_1 \oplus \ldots \oplus \mathbb{K}_m) \ltimes J(\mathcal{P}_A)$ where $J(\mathcal{P}_A)_A^2 = 0$, $J(\mathcal{P}_A)_L$ is a Lie $A$-algebra and $[\mathcal{P}, \mathbb{K}_1 \oplus \ldots \oplus \mathbb{K}_m]=0$.
    Let $e_i$ be the identity element of the field $\mathbb{K}_i$ and put $e=e_1+\ldots+ e_m$.  
    Define 
    $$\mathcal{L}_i = e_i J(\mathcal{P}_A) = \mathbb{K}_i J(\mathcal{P}_A) \textrm{ for $1\leq i\leq m$} \qquad  \textrm{ and } \qquad \mathcal{L}_0 = \left\{x\in J(\mathcal{P}_A): ex = 0 \right\}.$$

    We claim $J(\mathcal{P}_A) = \mathcal{L}_0 \oplus \mathcal{L}_1 \oplus \ldots \oplus \mathcal{L}_m$ as Poisson algebras. Indeed, for $x\in J(\mathcal{P}_A)$ we have
    $$x = (x-ex) + e_1x + \ldots + e_mx,$$
    where $(x-ex)\in \mathcal{L}_0$ and $e_ix \in \mathcal{L}_i$. Also, if $x_0  +\ldots + x_m = 0$ with $x_i \in \mathcal{L}_i$, then multiplying by $e_i$ gives $x_i = 0$ for every $i\geq 1$, and then $x_0 =0$. So the sum is direct as vector spaces. 
    
    Moreover, we note that each $\mathcal{L}_i$ is an ideal of $J(\mathcal{P}_A)_L$, because for all  $x, y\in J(\mathcal{P}_A)$ and every $1\leq i \leq m$ we have
    $[e_ix, y] = e_i[x,y] \in \mathcal{L}_i$. Also, for all $z\in \mathcal{L}_0$, we have $e[z, y] = [ez, y] = 0$, so $[z,y]\in \mathcal{L}_0$. Hence, for $i\neq j$ we have proven $[\mathcal{L}_i, \mathcal{L}_j] =0$. Together with $J(\mathcal{P}_A)_A^2=0$, this proves the previous claim. Here, we observe that the ideals $\mathcal L_i$ are Lie $A$-algebras over $\mathbb F$, as so is $J(\mathcal{P}_A)_L$.

    Now, notice that $ \mathbb{K}_i\mathcal{L}_j = 0$ and $[\mathbb{K}_i, \mathcal{L}_j] = 0$ for $i\neq j$. Therefore, we can write 
    $$\mathcal{P} = \mathcal{L}_0 \oplus (\mathbb{K}_1\ltimes \mathcal{L}_1) \oplus \ldots \oplus (\mathbb{K}_m\ltimes \mathcal{L}_m)$$
    as Poisson algebras.
    Since $J(\mathcal{P}_A)_A^2=0$ and
    $[\mathcal{P},\mathbb{K}_i]=0$, for
    $\lambda,\mu\in\mathbb{K}_i$ and
    $x,y\in\mathcal{L}_i$ we have
    $$
(\lambda,x)(\mu,y)
=
(\lambda\mu,\lambda y+\mu x),
\qquad
[(\lambda,x),(\mu,y)]
=
(0,[x,y]).
$$
 In particular, $\mathcal L_i$ is a Lie algebra over
    $\mathbb K_i$, because $[\mathbb K_i, \mathcal{L}_i]=0$. Additionally, every nilpotent $\mathbb K_i$-subalgebra
    of $\mathcal L_i$ is also a nilpotent $\mathbb F$-subalgebra of
    $J(\mathcal P_A)_L$, and hence is abelian. Thus $\mathcal L_i$ is
    a Lie $A$-algebra over $\mathbb K_i$.
    Hence, for $1\leq i\leq m$ we have $\mathbb{K}_i\ltimes\mathcal{L}_i    =\widehat{\mathcal{L}}_i^{\,\mathbb{K}_i},$ and the result follows.
\end{proof}

\begin{coro}
Let $\mathcal P$ be a finite-dimensional Poisson $A$-algebra over an
algebraically closed field $\mathbb F$ of characteristic different from
$2$. Then
$$\mathcal P
\cong
\mathcal L_0\oplus
\widehat{\mathcal L}_1^{\,\mathbb F}
\oplus\cdots\oplus
\widehat{\mathcal L}_m^{\,\mathbb F},$$
where $\mathcal L_0,\ldots,\mathcal L_m$ are finite-dimensional Lie
$A$-algebras over $\mathbb F$.
\end{coro}

}
\subsection{The maximal zero ideal} We next use the preceding results to identify a canonical zero
ideal of a Poisson $A$-algebra. 
In this section, we assume the field $\mathbb{F}$ has characteristic $p\neq 2, 3$. Define the sets
$$\mathcal{Q}(\mathcal{P}) = \left\{x\in \mathcal{P}: P_x^2 = Q_x^2 = 0\right\} \qquad \textrm{ and }\qquad \mathcal{Q}_L(\mathcal{P}) = \left\{x\in \mathcal{P}: Q_x^2 = 0\right\}.$$

\begin{lemma}\label{sqlm}
    Let $\mathcal{P}$ be a Poisson algebra. For $a\in \mathcal{Q}(\mathcal{P})$, the subalgebra $S_{\mathbb{F}a}$ is nilpotent. Moreover, if $\mathcal{P}$ is a Poisson $A$-algebra, then $a^2 = 0$. 
\end{lemma}
\begin{proof}
    Clearly, we have $S_{\mathbb{F}a} = \mathbb{F} a + \mathbb{F} a^2$, which is nilpotent. 
\end{proof}

We write $\langle a, b\rangle$ for the subalgebra of $\mathcal{P}_L$ generated by $a$ and $b$. 

\begin{lemma}\label{lmab}
    Let $\mathcal{P}$ be a Poisson $A$-algebra. For $a, b\in \mathcal{Q}(\mathcal{P})$, denote $U:= \langle a, b\rangle$. Then $S_{U}$ is a zero subalgebra of $\mathcal{P}$. In particular, we have $ab = 0$ and $[a,b] = 0$.
\end{lemma}
\begin{proof}
    We have $U = \mathbb{F} a + \mathbb{F} b + \mathbb{F} c$, where $c = [a, b]$.
    By Lemma~\ref{sqlm}, we have $a^2= b^2 = 0$. Therefore, $ac = \frac{1}{2}[a^2, b] = 0$ and $bc = - \frac{1}{2}[a, b^2] = 0$. 
    Also, $c^2 = [a,b]c = [a,bc] - [a,c]b = 0$. Hence, $U^3_A = 0$ and $S_U =  U + \mathbb{F} ab$. This proves $S_U$ is  associative nilpotent. By Proposition~\ref{propnil}, $U$ is a nilpotent subalgebra of $(S_U)_L$ and $(S_U)_A$ is nilpotent, so $S_U$ is nilpotent. Since $\mathcal{P}$ is a Poisson $A$-algebra, $S_U$ is a zero algebra.
\end{proof}

The following useful relations are attributed to Kostrikin in \cite{PS90}.

\begin{lemma}\label{eqKostrikin}
    Let $\mathcal{P}$ be a Poisson algebra over a field of characteristic $p\neq 2$. If $a\in \mathcal{Q}_L(\mathcal{P})$ and $x\in \mathcal{P}$, then
    \begin{enumerate}
        \item  $Q_aQ_xQ_a=0$,
        \item  $[Q_a, Q_x]^2 = - Q_a Q_x^2 Q_a$ and
        \item  if moreover $p\neq 3$, then $Q_aQ_x^2 Q_aQ_x^2 Q_a = 0$.
    \end{enumerate}
\end{lemma}
\begin{proof}
    Since $[Q_a, [Q_x, Q_a]]= Q_{[a, [x, a]]}=0$, we have
    $$[Q_a, [Q_x, Q_a]] = Q_a Q_x Q_a - Q_a Q_a Q_x - Q_x Q_a Q_a + Q_a Q_x Q_a = 2 Q_a Q_x Q_a.$$
    Hence, we obtain $Q_aQ_xQ_a=0$. For the second equation, we obtain the first equation
    $$[Q_a, Q_x]^2 = Q_aQ_xQ_aQ_x - Q_aQ_xQ_xQ_a - Q_xQ_aQ_aQ_x + Q_xQ_aQ_xQ_a= - Q_aQ_xQ_xQ_a.$$
    Similarly, the third equation follows from $Q_a Q_{[x, [x, [x, [x, a]]]]}Q_a = 0$, since
    $$Q_a Q_{[x, [x, [x, [x, a]]]]}Q_a = Q_a [Q_x, [Q_x, [Q_x, [Q_x, Q_a]]]] Q_a = 6 Q_a Q_x^2 Q_a Q_x^2 Q_a.$$
    This proves the last equation when the characteristic of the base field is neither $2$ or $3$. 
\end{proof}

\begin{propo}
    Let $\mathcal{P}$ be a Poisson $A$-algebra. Then $\mathcal{Q}(\mathcal{P})$ is a zero ideal.
\end{propo}
\begin{proof}
    Let $a\in \mathcal{Q}(\mathcal{P})$ and $x\in \mathcal{P}$. First, we show $[\mathcal{Q}(\mathcal{P}), \mathcal{P}]\subseteq \mathcal{Q}(\mathcal{P})$, i.e.   
    $Q_{[a, x]}^2=0$ and $P_{[a, x]}^2=0$.  Notice that $a^2 = 0$ implies that
    \begin{equation}\label{eqt}
        P_aQ_a = 0, \qquad Q_aP_a = 0, \qquad P_aQ_xP_a =0.
    \end{equation}

    Let $x, z\in \mathcal{P}$. Setting $b = [a, z]$ and $c = [[a,x], x]$, we obtain that the subalgebra $\langle b, c\rangle$, is nilpotent, as in \cite[Proposition 1]{PS90}. Indeed, we have
    $$[b, [b, c]] = Q_b^2(c) = Q_{[a,z]}^2(c) = [Q_a, Q_z]^2(c) = - Q_a Q_z^2 Q_a (c) = 0,$$
    $$[c, [c, b]] = Q_c^2(b) = Q_{[[a,x], x]}^2Q_a(z) = [[Q_a, Q_x], Q_x]^2Q_a(z) = Q_aQ_x^2Q_aQ_x^2Q_a(z) = 0.$$
    Write $U = \langle b, c\rangle$. 
    We prove that $S_U$ is associative nilpotent. Note that $U = \mathbb{F} b + \mathbb{F} c + \mathbb{F} [b,c] $. Now
    $$2 b^2 = 2 [a, z] [a, z] = [a, z[a,z]] - 2 [a, [a, z]]z  = [a, [a,z^2]] = 0.$$
Put $d=[a,x]$. Then $d^2=0$ as above. Also, $[b, c] [b, c] = [b[b, c], c] - b [[b,c], c]=0$.
    $$c^2 = [d, x][d, x] = [d[d,x],x]-d[[d,x],x]= -d[[d,x],x], \textrm{ so} $$
    $$2c^3=-2d[d,x][[d,x],x]=[d^2,x][[d,x],x]=0.$$
    
Hence $S_U = \mathbb{F} b + \mathbb{F} c + \mathbb{F} c^2 + \mathbb{F} [b,c] + \mathbb{F} [b,c^2]+\mathbb{F} bc +\mathbb{F}bc^2$, which is nilpotent, since $(S_U)_A^4= 0$.

    Therefore, $U$ is a nilpotent subalgebra of $(S_U)_L$ and $(S_U)_A$ is nilpotent, so $S_U$ is nilpotent and a zero algebra. In particular, we have $[b, c]=0$. It follows
    \begin{equation}
        \begin{split}
            Q_a Q_x^2 Q_a & = Q_a Q_x Q_aQ_x + Q_a Q_x Q_{[x, a]} = Q_a Q_{[x, a]} Q_x + Q_a Q_{[x,[x, a]]}\\
            & = Q_a Q_x Q_aQ_x  - Q_a Q_a Q_xQ_x + Q_a Q_{[x,[x, a]]} = Q_a Q_{[x,[x, a]]}.
        \end{split}
    \end{equation}
    But $Q_a Q_{[x,[x, a]]} = 0$, since $[b, c]=0$. Therefore, $Q_{[a, x]}^2 =  [Q_a, Q_x]^2 = 0$, by 
    Lemma~\ref{eqKostrikin}.   
    Also, we obtain $P_{[a, x]}^2= P_{[a, x]^2} = 0$, because $[a,x]^2 = [a, x[a,x]] - [a, [a, x]]x = 0$. 

    Now, we prove $\mathcal{Q}(\mathcal{P}) \mathcal{P}\subseteq \mathcal{Q}(\mathcal{P})$.
    Therefore, by $[\mathcal{Q}(\mathcal{P}), \mathcal{P}]\subseteq \mathcal{Q}(\mathcal{P})$, Lemma~\ref{lmab} and equation~\eqref{eqt}, we deduce 
    $$Q_{ax}^2 = P_aQ_xP_aQ_x + P_xQ_aP_aQ_x + P_aQ_xP_xQ_a + P_xQ_aP_xQ_a = 0.$$
    Moreover, $P_{ax}^2 = P_{a^2x^2} = 0$. Hence, $\mathcal{Q}(\mathcal{P}) \mathcal{P}\subseteq \mathcal{Q}(\mathcal{P})$. 
    
    Finally, we show $\mathcal{Q}(\mathcal{P})$ is a vector space. For $a,b\in \mathcal{Q}(\mathcal{P})$, we have 
    $$P_{a+b}^2 = P_{a^2} + 2P_{ab} + P_{b^2} = 0 \qquad \textrm{ and }\qquad Q_{a+b}^2 = Q_a^2 + Q_aQ_b + Q_bQ_a + Q_b^2 = 0,$$
    since $[\mathcal{Q}(\mathcal{P}), \mathcal{P}]\subseteq \mathcal{Q}(\mathcal{P})$. From here, we derive that $\mathcal{Q}(\mathcal{P})$ is a vector space. Then Lemma~\ref{lmab}, together with $[\mathcal{Q}(\mathcal{P}), \mathcal{P}]\subseteq \mathcal{Q}(\mathcal{P})$ and $\mathcal{Q}(\mathcal{P}) \mathcal{P}\subseteq \mathcal{Q}(\mathcal{P})$ imply $\mathcal{Q}(\mathcal{P})$  is a zero ideal. 
\end{proof}

\begin{coro}
    Let $\mathcal{P}$ be a Poisson $A$-algebra. Then $\mathcal{Q}(\mathcal{P})$ is the unique maximal zero ideal.
\end{coro}
\begin{proof}
By the preceding proposition, $\mathcal Q(\mathcal P)$ is a zero ideal. If $I$ is any zero ideal and $a\in I$, then, for every $x\in\mathcal P$, we have
$$P_a^2(x)=a(ax)=0,\qquad Q_a^2(x)=[a,[a,x]]=0,$$
since $ax,[a,x]\in I$ and $I$ is a zero algebra. Thus $I\subseteq\mathcal Q(\mathcal P)$.
\end{proof}

\subsection{Solvable Poisson $A$-algebras} We conclude the section with an application to solvable Poisson
$A$-algebras, showing that some classical properties of solvable Lie
$A$-algebras admit a natural Poisson analogues.
 
\begin{lemma}\label{cartan} If $C$ is a Cartan subalgebra of a Poisson $A$-algebra $\mathcal{P}$ in which $\mathcal{P}_A$ is nilpotent, and $\mathcal{P}=C\dot{+} P_1$ is the Fitting decomposition of $\mathcal{P}_L$ relative to $Q_C$, then $C\cdot \mathcal{P}_1\subseteq \mathcal{P}_1$. Moreover, if $\mathcal{P}_1$ is a Lie ideal of $\mathcal{P}_L$, it is an ideal of $\mathcal{P}$.
\end{lemma}
\begin{proof} We have that $[C,C]=0$, by Proposition~\ref{lie2}, and $[C,\mathcal{P}_1]=\mathcal{P}_1$, so $$C\cdot \mathcal{P}_1=C\cdot [C,\mathcal{P}_1]\subseteq [C,C\cdot \mathcal{P}_1]\subseteq [C,C+\mathcal{P}_1]\subseteq \mathcal{P}_1.$$ If $\mathcal{P}_1$ is a Lie ideal of $\mathcal{P}_L$, then $[\mathcal{P}_1,\mathcal{P}_1]\subseteq \mathcal{P}_1$, so $$\mathcal{P}_1\cdot \mathcal{P}_1=\mathcal{P}_1\cdot [C,\mathcal{P}_1]\subseteq [C\cdot \mathcal{P}_1,\mathcal{P}_1]+C\cdot [\mathcal{P}_1,\mathcal{P}_1]\subseteq \mathcal{P}_1,$$ and $\mathcal{P}_1$ is an ideal of $\mathcal{P}$.
\end{proof}

It follows from Corollary \ref{cor33} that the results in \cite{Aalg} hold for solvable Poisson algebras over a field of characteristic different from two. However, the results generally do not require this restriction on the characteristic of the field, despite the fact that the example given in the remark after Corollary \ref{cor33} is solvable. The proofs only require minor changes to those employed for Lie algebras. We'll demonstrate this by proving just one of them.

\begin{theor} Let $\mathcal{P}$ be a finite-dimensional  solvable Poisson $A$-algebra. Then $Z(\mathcal{P})\cap \mathcal{P}^2=0$.
\end{theor}
\begin{proof}  Let $\mathcal{P}$ be a minimal counter-example and let $A$ be a minimal ideal of $\mathcal{P}$. Then $A$ is a zero algebra, $\mathcal{P}\cdot A=0$ and $[\mathcal{P},A]=A$ or $0$. Now $$\frac{Z(\mathcal{P})+A}{A}\cap \frac{\mathcal{P}^2+A}{A}\subseteq Z\left(\frac{\mathcal{P}}{A}\right)\cap \left(\frac{\mathcal{P}}{A}\right)^2=\frac{A}{A}.$$ If $[\mathcal{P},A]=A$, then $A\subseteq \mathcal{P}^2$ and $Z(\mathcal{P})\cap \mathcal{P}^2\subseteq A$; if $[\mathcal{P},A]=0$, then $A\subseteq Z(\mathcal{P})$ and  $Z(\mathcal{P})\cap \mathcal{P}^2\subseteq A$ again. Hence, if $A_1$, $A_2$ are minimal ideals of $\mathcal{P}$, $Z(\mathcal{P})\cap \mathcal{P}^2\subseteq A_1\cap A_2=0$. It follows that $\mathcal{P}$ is monolithic with monolith $A=Z(\mathcal{P})=\mathbb{F}z$, by the minimality.
\par

Now let $M$ be a maximal ideal of $\mathcal{P}$. Then $Z(M)\cap M^2=0$ by the minimality of $\mathcal{P}$, so $Z(\mathcal{P})\not\subseteq M^2$, whence $M^2=0$. It follows that $\mathcal{P}=M+\mathbb{F}x$ for some $x\in \mathcal{P}$ and $M^2=0$. Let $\mathcal{P}=\mathcal{P}_0\dot{+}\mathcal{P}_1$ be the Fitting decomposition relative to  $Q_x$. Then $\mathcal{P}_1\subseteq M$ and $[\mathcal{P}_0,\mathcal{P}_1]+\mathcal{P}_0\cdot \mathcal{P}_1\subseteq \mathcal{P}_1$ as in Lemma \ref{cartan}, so $\mathcal{P}_1$ is an ideal of $\mathcal{P}$. If $\mathcal{P}_1\neq 0$, then $Z(\mathcal{P})\subseteq \mathcal{P}_0\cap \mathcal{P}_1=0$, so $\mathcal {P}_1=0$. It follows that $\mathcal{P}_L$ is nilpotent and hence that $\mathcal{P}$ is nilpotent. Since $\mathcal{P}$ is an $A$-algebra, this means that, $\mathcal{P}$ is a zero algebra, a contradiction.
\end{proof}

\end{document}